\documentclass[aip,jmp]{revtex4-1}
\usepackage{graphicx}
\usepackage{dcolumn}
\usepackage{bm}
\usepackage{amssymb}
\usepackage{amsmath,amsthm}
\usepackage{mathptmx}
\usepackage{tikz}
\usepackage[all]{xy}
\usetikzlibrary{arrows}
\usetikzlibrary{decorations.markings}

\makeatletter

\newcommand{\Rmnum}[1]{\expandafter\@slowromancap\romannumeral #1@}
\makeatother
\newfam\msbfam
\font\twelmsb=msbm10 at 12pt 
\font\sevenmsb=msbm10 at 7pt \font\fivemsb=msbm10 at 5pt
\textfont\msbfam=\twelmsb
\scriptfont\msbfam=\sevenmsb \scriptscriptfont\msbfam=\fivemsb

\newtheorem{theorem}{Theorem}[section]
\newtheorem{definition}[theorem]{Definition}
\newtheorem{proposition}[theorem]{Proposition}

\def\N{\mathbb{N}}

\def\Z{\mathbb{Z}}

\def\CC{\mathcal{C}}

\def\A{\mathcal{A} }

\allowdisplaybreaks
\numberwithin{equation}{section}

\begin{document}
\preprint{}

\title{Quantum cluster algebra structure on the finite dimensional representations of $U_q(\widehat{sl_{2}})$}

\author{Hai-Tao Ma, Yan-Min Yang, Zhu-Jun Zheng}
\email{zhengzj@scut.edu.cn}
\affiliation{Department of Mathematics\\
 South China University of
 Technology\\ Guangzhou 510641, P. R. China.}

\begin{abstract}
In this paper, 
we give a quantum cluster algebra structure on the deformed Grothendieck ring of $\CC_{n}$, where $\CC_{n}$ is  a full subcategory of finite dimensional representations of $U_q(\widehat{sl_{2}})$ defined in section II.
\end{abstract}


\maketitle

\section{Introduction}

Cluster algebras were introduced by Fomin and Zelevinsky\cite{FZ0}, and they played more and more important role in representation theory and other areas. In 2010, Hernandez and Leclerc \cite{HL0} posted a conjecture on the monoidal categorification of cluster algebras \cite{FZ0} and gave a model of monoidal category for certain cluster algebras. They  proved the case $n = 1$ for the Lie algebras of type $AD$\cite{HL0,HL1}. Almost at the same time, Nakajima\cite{Na1} used the theory of perverse sheaves and q,t-characters to prove the case $n = 1$ for the Lie algebras of type $ADE$. In 2012, Qin \cite{Qin} generalized Nakajima's geometric approach and obtained monoidal categorifications of the cluster algebras associated to any acyclic quiver. Recently, Yang and Zheng\cite{YZ1} proved the conjecture for the case $A_3, n = 2$.

In 2010, Nakajima\cite{Na1} asked that if there is a quantum cluster algebra structure on the deformed Grothendieck ring of full subcategory of the finite dimensional representation of the quantum loop algebra. In 2012, Qin and Yoshiyuki Kimura \cite{QY1} obtained a deformed monoidal categorifications of acyclic quantum cluster algebras with specific coefficients through geometric approach.

    In this paper, by the algebraic definition of q, t-characters given by Hernandez\cite{H0}, we want to give an answer of Nakajima's problem of $U_q(\widehat{sl_{2}})$ case. That is , we proved that for any $n \in \N$, there is a quantum cluster algebra structure on the deformed Grothendieck ring of the full subcategory $\CC_{n}$.

    The paper is organized as follows. In Section 2, we present necessary definitions and
facts of quantum cluster algebras.
In Section 3, we give the main result and the detail of the proof.

\section{Definitions and Notations}\label{s2}
In this section, we give a simplified introduction to the theory of quantum cluster algebras \cite{BZ0} and finite dimensional representations of quantum affine algebras \cite{CP0,CP1,H0}.
\subsection{Quantum cluster algebra}

\begin{definition}
Let $\widetilde{B}$ be an $m \times n$ integer matrix with rows labeled by $[1,m]$ and columns labeled by an n-element subset $\textbf{ex} \subset [1, m]$. Let $\Lambda$ be a skew-symmetric $m \times m$ integer matrix with rows and columns labeled by $[1, m]$. We say that a pair $(\Lambda, \widetilde{B})$ is compatible if, for every $j \in \textbf{ex}$ and $i \in [1, m]$, we have
$$\sum_{ k = 1}^{m}b_{kj}\lambda_{ki} = \delta_{ij}d_{j}$$
for some integers $d_{j}(j \in \textbf{ex}).$
\end{definition}

Fix an index $k \in \textbf{ex}$ and a sign $\varepsilon \in \{\pm 1\},$ the matrix $\widetilde{B}^{'} = \mu_{k}(\widetilde{B})$ is defined as
$$\widetilde{B}^{'} = E_{\varepsilon}\widetilde{B}F_{\varepsilon},$$
where $E_{\varepsilon}$ is the $m \times m$ matrix with entries

$$ e_{ij} =\left\{\begin{array}{ll}
\delta_{ij} & \text{if $j \neq  k$,} \\[.05in]
-1
 & \text{if $i = j = k$,
 }  \\[.05in]
 max(0, -\varepsilon b_{ik}) &\text{if $i\neq j = k$,}
 \end{array}   \right.
$$
$F_{\varepsilon}$ is the $n \times n$ matrix with rows and columns labeled by
$\textbf{ex}$ with  entries 

$$ f_{ij} =\left\{\begin{array}{ll}
\delta_{ij} & \text{if $i \neq  k$,} \\[.05in]
-1
 & \text{if $i = j = k$,
 }  \\[.05in]
 max(0, \varepsilon b_{ik}) &\text{if $i = k \neq j$.}
 \end{array}   \right.
$$

Set $\Lambda^{'} = E_{\varepsilon}^{T} \Lambda E_{\varepsilon}.$
\begin{definition}
Let $(\Lambda, \widetilde{B})$ be a compatible pair, and $k \in \textbf{ex}$. We say that the compatible pair $(\Lambda^{'}, \widetilde{B}^{'})$ is obtained from $(\Lambda, \widetilde{B})$ by the mutation in direction k.
\end{definition}

Let $L$ be a lattice of rank m with a skew-symmetric bilinear form $\Lambda : L \times L \rightarrow \Z.$
The based quantum torus associated with $L$ is the $\Z[q^{\pm1/2}]$-algebra $T = T(\Lambda)$ with a distinguished $\Z[q^{\pm 1/2}]$-basis $\{X^{e}| e \in L\}$
and the multiplication given by
$X^{e}X^{f} = q^{\Lambda(e, f)/2}X^{e + f},$
$X^{0} = 1, (X^{e})^{-1} = X^{-e}.$
For a skew-field $F$ of $T$, a toric frame in $F$ is a map $M : \Z^{m} \rightarrow F \backslash \{0\}$ defined as
$M(c) = \varphi(X^{\eta(c)}),$
where $\varphi$ is an automorphism of $F$, and $\eta : \Z^{m} \rightarrow L$ is an isomorphism of lattices.

\begin{definition}
A quantum seed is a pair $(M, \widetilde{B})$, where
M is a toric frame in F; $\widetilde{B}$ is an $m \times n$ integer matrix with rows labeled by $[1, m]$ and columns labeled by an n-element subset $\textbf{ex} \subset [1, m];$ the pair$(\Lambda_{M}, \widetilde{B})$ is a compatible pair.
\end{definition}

Let $(M,\widetilde{B})$ be a quantum seed, and $(M^{'}, \widetilde{B}^{'})$ be the quantum seed obtained by the mutation in direction $ k \in \textbf{ex}$. For $i \in [1, m]$, let $X_{i} = M(e_{i})$ and $X_{i}^{'} = M^{'}(e_{i})$. Then $X_{i}^{'} = X_{i}$ for $i \neq k$, and $X_{k}^{'}$ is given by the following quantum analog of the exchange relation:
$$X_{k}^{'} = M(-e_{k} + \sum_{b_{ik} > 0}b_{ik}e_{i}) + M(- e_{k} + \sum_{b_{ik} < 0}b_{ik}e_{i}).$$

Two quantum seeds are mutation-equivalent if they can be obtained from each other by a sequence of quantum seed mutations. For a quantum seed $(M,\widetilde{B}),$ we denote by $\widetilde{\textbf{X}} = (X_{1}, X_{2},\cdots,X_{m})$ the corresponding "free generating set" in $F$ given by $X_{i} = M(e_{i})$. We call the subset$\textbf{X} = \{X_{j} | j \in \textbf{ex}\} \subset \widetilde{\textbf{X}}$ the cluster of the quantum seed $(M, \widetilde{B}),$
and set $\textbf{C} = \widetilde{\textbf{X}} - \textbf{X}$.

\begin{definition}
Let $S$ be a mutation-equivalence class of quantum seeds in $F$. The quantum cluster algebra associated with $S$ is the $\Z[q^{\pm1/2}]$-subalgebra of the ambient skew-field F, generated by the union of clusters of all seeds in $S$, together with the element of $\textbf{C}$.
\end{definition}

\subsection{Quantum loop algebra}
$U_q(\widehat{sl_{2}})$ is the corresponding quantum affine algebra with parameter $q\in C^*$ not a root of unity. $\CC$ is the category of finite-dimensional $U_q(\widehat{sl_{2}})$-representations of type 1.
For $\ell\in \Z_{\geq0}$,   $\CC_\ell$ be the full subcategory of $\CC$:
 for any $V$ of $\CC_{\ell}$, the roots of the Drinfeld polynomials of
every composition factor of $V$ belong to $\{q^{-2k}\mid 0\le k \le \ell\}$.

Denote the Grothendieck ring of $\CC_{\mathbb{\ell}}$ by $R_{\mathbb{\ell}}$, then $R_{\mathbb{\ell}}=\mathbb{Z}\big[[W_{1,2k}]_{i\in I,0\leq k\leq \ell}\big]$, where  $W_{1,2k}$ are Kirillov-Reshetikhin modules with the highest l-weights
$m_{1,2k} = Y_{q^{2k}}.$

\begin{theorem}\cite{H0}
In the $sl_{2}$ case, the deformed Grothendieck ring $Rep_{t,n}$ is a $\Z[q^{\pm1/2}]$-algebra generated by $[W_{1,0}], [W_{1,2}], \cdot\cdot\cdot, [W_{1,2n}]$ with relations
$$[W_{1,l_{1}}]\ast[W_{1,l_{2}}] = q^{r}[W_{1,l_{2}}]\ast[W_{1,l_{1}}], if \ l_{1} \geq l_{2} \ and \ l_{1} \neq l_{1} + 2, where \ r = (-1)^{(l_{1} - l_{2})/2}.$$
$$[W_{1,l}]\ast[W_{1,l - 2}] = q^{-1}[W_{1,l-2}]\ast[W_{1,l}] + (1 - q^{-1}). $$
\end{theorem}

\section{Main Results}\label{s3}

In this section, we construct a quantum cluster algebra structure on the deformed Grothendieck ring  of $\CC_{n}$  for the Lie algebra $sl_2$. To be easily reading, we give an example for a special case.

The $\Gamma_{n}$ is the quiver
\[
\begin{matrix} n+1& \rightarrow & n &\rightarrow & \cdots&\rightarrow & 1
\end{matrix},
\]

where
$n+1$ is the frozen point.

 Let $\widetilde{B_{n}}$ is a matrix correspond to $\Gamma_{n}$ , i.e
$$\widetilde{B_{n}} = \left(
                        \begin{array}{ccccc}
                          0 & -1 & 0 & \cdot\cdot\cdot & 0 \\
                          1 &0 & -1 & \cdot\cdot\cdot & 0 \\
                          0 & 1 & 0 & \cdot\cdot\cdot & 0 \\
                          0 & 0& 1 & \cdot\cdot\cdot & 0 \\
                          \vdots &\vdots & \vdots  & \vdots  & \vdots  \\
                          0 & 0 & 0 & \cdot\cdot\cdot & 1 \\
                        \end{array}
                      \right)_{(n+1,n)}
$$

Let $\Lambda_{n} = (\lambda_{ij})_{(n+1,n+1)}$  be the skew matrix  defined by

$$ \lambda_{ij} =\left\{\begin{array}{ll}
(-1)^{i+j+1} & \text{if $j > i$ and i is odd,} \\[.05in]
0
 & \text{if $j > i$ and i is even ,
 }  \\[.05in]
 - \lambda_{ji} &\text{if $j < i$.}
 \end{array}   \right.
$$
That is
$$\Lambda = \left(
    \begin{array}{ccccccc}
      0 & -1 & 0 & -1 & 0 & -1 & \cdot\cdot\cdot \\
      1 & 0 & 0 & 0 & 0 & 0 & \cdot\cdot\cdot \\
      0 & 0 & 0 & -1 & 0 & -1 & \cdot\cdot\cdot \\
      1 & 0 & 1 & 0 & 0 & 0 & \cdot\cdot\cdot \\
      \vdots& \vdots & \vdots & \vdots & \vdots & \vdots & \vdots\\
    \end{array}
  \right)_{(n+1,n+1)}.
$$
It is easy to see $(\Lambda_{n},\widetilde{B}_{n})$ is compatible.

Let $\A_{n} = \A(\Gamma_{n})$  be the quantum cluster algebra associated with a  pair$(\Lambda_{n},\widetilde{B}_{n})$. The ambient field $F$ of fractions of the quantum torus with generators $Y_{1}, Y_{2}, \cdots, Y_{n+1}$ satisfying relations $Y_{i}Y_{j} = q^{\lambda_{ij}}Y_{j}Y_{i}$.  
For $0 \leq i \leq n-1$, let $X_{2i}=\mu_{n-i}(Y_{n-i})$, and $X_{2n} = Y_{1}$.

\begin{proposition}\label{equithm}
 $\A_n $ as a  $\Z[q^{\pm1/2}]$ algebra can be generated by $\{X_{2i} \ | \ 0 \leq i \leq n\}.$
\end{proposition}

\begin{proof}
Since quantum clusters are mutation equivalence if and only if the correspondence clusters are mutation equivalence,
the number of the quantum cluster variables is equal to the number of the cluster variables. The cluster variable is one to one correspondence to the set $\{[W_{i,j}] | 1\leq i \leq n+1 , 0 \leq j \leq 2n - 2i +2, {\rm where }~ j~ {\rm is ~ even} \}$\cite{L1}.

By the definition of the quantum cluster algebra, we have the following relations:
$$X_{2n-2}X_{2n} = q^{1/2}Y_{2} + 1;$$
$$X_{2n-2i}Y_{i} = q^{-1/2}Y_{i-1} + Y_{i+1},{\rm if~ }i~ {\rm is~ even};$$
$$X_{2n-2i}Y_{i} = Y_{i-1} + q^{1/2}Y_{i+1},{\rm if} ~ i~ {\rm is~ odd};$$

By the above relations, it is easy to see that $Y_{i}(1 \leq i \leq n) \in \Z[X_{2i}(0 \leq i \leq n+1),q^{\pm1/2}]$.
For example, \ $Y_{3} = X_{2n-4}Y_{2} + q^{-1/2}Y_{1} = q^{-1/2}(X_{2n-4}X_{2n-2}X_{2n} - X_{2n-4} + X_{2n}).$

 Set $\mu = \mu_{n}\mu_{n-1}\cdots\mu_{1}$, where $\mu_{i}$ is the mutation at the direction i, $(Y_{1}^{'}, Y_{2}^{'},\cdot\cdot\cdot,Y_{n}^{'},Y_{n+1}) = \mu(Y_{1}, Y_2, \cdots,Y_{n+1})$ , $Y_{i}^{''} = \mu_{i}(Y_i^{'})$ . By the mentioned above and the T-system, $(Y_{1}^{'}, Y_{2}^{'},\cdot\cdot\cdot,Y_{n}^{'},Y_{n+1})$ is correspondence to  $(W_{1,2n-2}, W_{2,2n-4}, \cdot\cdot\cdot, W_{n,0}, W_{n+1,0})$.  So we have $Y_{i}^{''} = X_{2n-2-2i}(1 \leq i \leq n-1)$, \ $Y_{1}^{'} = X_{2n-2}$. Similarly, we have
$Y_{i}^{'}(1 \leq i \leq n) \in \Z[X_{2i}(0 \leq i < n),q^{\pm1/2}].$

By the induction,\ we have that any quantum cluster variable belongs to $Z[X_{2i}(0 \leq i \leq n),q^{\pm1/2}].$
\end{proof}

\begin{theorem}\emph{}\label{thm}
The map
\[
\begin{array}{ccc}
  X_{i} \mapsto [W_{1,i}]
\end{array}
\]
extends to a ring isomorphism $\iota_{n}$  from the quantum cluster algebra $\A_{n}$
to the deformed Grothendieck ring $Rep_{t,n}$ of $\CC_n$.
\end{theorem}
\begin{proof}
First, if i is odd,
\[
\begin{array}{ll}
  & X_{2(n - i)}X_{2(n - i - 1)} \\
  = & (Y_{i - 1} + q^{1/2}Y_{i+1})Y_{i}^{-1}(q^{-1/2}Y_{i} + Y_{i + 2})Y_{i+1}^{-1} \\
  = & q^{-1/2}Y_{i - 1}Y_{i+1}^{-1} + 1 + Y_{i-1}Y_{i}^{-1}Y_{i+2}Y_{i+1}^{-1} + q^{-1/2}Y_{i}^{-1}Y_{i+2}.
\end{array}
\]

\[
\begin{array}{ll}
   & X_{2(n - i - 1)}X_{2(n - i)} \\
  = & (q^{-1/2}Y_{i} + Y_{i + 2})Y_{i+1}^{-1}(Y_{i - 1} + q^{1/2}Y_{i+1})Y_{i}^{-1} \\
  = & q^{1/2}Y_{i - 1}Y_{i+1}^{-1} + 1 + qY_{i-1}Y_{i}^{-1}Y_{i+2}Y_{i+1}^{-1} + q^{1/2}Y_{i}^{-1}Y_{i+2}.
\end{array}
\]

So we have$X_{2(n - i)}X_{2(n - i - 1)} = q^{-1}X_{2(n - i - 1)}X_{2(n - i)} + (1 - q^{-1})$.

If i is even,
\[
\begin{array}{ll}
   & X_{2(n - i)}X_{2(n - i - 1)}\\
  = & (q^{-1/2}Y_{i - 1} + Y_{i+1})Y_{i}^{-1}(Y_{i} + q^{1/2}Y_{i + 2})Y_{i+1}^{-1} \\
  = & q^{-1/2}Y_{i - 1}Y_{i+1}^{-1} + 1 + Y_{i-1}Y_{i}^{-1}Y_{i+2}Y_{i+1}^{-1} + q^{-1/2}Y_{i}^{-1}Y_{i+2}.
\end{array}
\]

\[
\begin{array}{ll}
   & X_{2(n - i - 1)}X_{2(n - i)} \\
  = & (Y_{i} + q^{1/2}Y_{i + 2})Y_{i+1}^{-1}(q^{-1/2}Y_{i - 1} + Y_{i+1})Y_{i}^{-1} \\
  = &  q^{1/2}Y_{i - 1}Y_{i+1}^{-1} + 1 + qY_{i-1}Y_{i}^{-1}Y_{i+2}Y_{i+1}^{-1} + q^{1/2}Y_{i}^{-1}Y_{i+2}.
\end{array}
\]

So we have$X_{2(n - i)}X_{2(n - i - 1)} = q^{-1}X_{2(n - i - 1)}X_{2(n - i)} + (1 - q^{-1})$.

Secondly,
\[
\begin{array}{ll}
   & X_{2(n - i)}X_{2(n - j)}\\
  = & (Y_{i - 1} + q^{1/2}Y_{i+1})Y_{i}^{-1}(Y_{j} + q^{1/2}Y_{j + 1})Y_{j}^{-1}\\
  = & Y_{i - 1}Y_{i}^{-1}Y_{j-1}Y_{j}^{-1} + q^{1/2}Y_{i+1}Y_{i}^{-1}Y_{j-1}Y_{j}^{-1}
 + q^{1/2}Y_{i - 1}Y_{i}^{-1}Y_{j+1}Y_{j}^{-1} + qY_{i+1}Y_{i}^{-1}Y_{j+1}Y_{j}^{-1}.
\end{array}
\]

if i and j is odd and $i < j$;
\[
\begin{array}{ll}
   & X_{2(n - j)}X_{2(n - i)} \\
 = & Y_{j - 1} + q^{1/2}Y_{j+1})Y_{j}^{-1}(Y_{i} + q^{1/2}Y_{i + 1})Y_{i}^{-1} \\
  = & Y_{j-1}Y_{j}^{-1}Y_{i - 1}Y_{i}^{-1} + q^{1/2}Y_{j-1}Y_{j}^{-1}Y_{i+1}Y_{i}^{-1}
 + q^{1/2}Y_{j+1}Y_{j}^{-1}Y_{i - 1}Y_{i}^{-1} + qY_{j+1}Y_{j}^{-1}Y_{i+1}Y_{i}^{-1} \\
  = & q^{-1}(Y_{i - 1}Y_{i}^{-1}Y_{j-1}Y_{j}^{-1} + q^{1/2}Y_{i+1}Y_{i}^{-1}Y_{j-1}Y_{j}^{-1}
 + q^{1/2}Y_{i - 1}Y_{i}^{-1}Y_{j+1}Y_{j}^{-1} + qY_{i+1}Y_{i}^{-1}Y_{j+1}Y_{j}^{-1}) \\
  = & q^{-1}X_{2(n - i)}X_{2(n - j)} \\
  = & q^{(-1)^{j - i + 1}}X_{2(n - i)}X_{2(n - j)}.
\end{array}
\]
 If i is odd,j is even and $i +1 < j$;
 \[
 \begin{array}{ll}
    & X_{2(n - j)}X_{2(n - i)} \\
   = & (Y_{j - 1} + q^{1/2}Y_{j+1})Y_{j}^{-1}(Y_{i} + q^{1/2}Y_{i + 1})Y_{i}^{-1} \\
   = & Y_{j-1}Y_{j}^{-1}Y_{i - 1}Y_{i}^{-1} + q^{1/2}Y_{j-1}Y_{j}^{-1}Y_{i+1}Y_{i}^{-1}
 + q^{1/2}Y_{j+1}Y_{j}^{-1}Y_{i - 1}Y_{i}^{-1} + qY_{j+1}Y_{j}^{-1}Y_{i+1}Y_{i}^{-1} \\
   = & q(Y_{i - 1}Y_{i}^{-1}Y_{j-1}Y_{j}^{-1} + q^{1/2}Y_{i+1}Y_{i}^{-1}Y_{j-1}Y_{j}^{-1}
 + q^{1/2}Y_{i - 1}Y_{i}^{-1}Y_{j+1}Y_{j}^{-1} + qY_{i+1}Y_{i}^{-1}Y_{j+1}Y_{j}^{-1}) \\
   = & qX_{2(n - i)}X_{2(n - j)} \\
   = &q^{(-1)^{j - i + 1}}X_{2(n - i)}X_{2(n - j)}
 \end{array}
 \]

 If i and j are even,and $i < j$;

 \[
 \begin{array}{ll}
    & X_{2(n - j)}X_{2(n - i)} \\
   = & (Y_{j - 1} + q^{1/2}Y_{j+1})Y_{j}^{-1}(Y_{i} + q^{1/2}Y_{i + 1})Y_{i}^{-1} \\
   = & Y_{j-1}Y_{j}^{-1}Y_{i - 1}Y_{i}^{-1} + q^{1/2}Y_{j-1}Y_{j}^{-1}Y_{i+1}Y_{i}^{-1}
 + q^{1/2}Y_{j+1}Y_{j}^{-1}Y_{i - 1}Y_{i}^{-1} + qY_{j+1}Y_{j}^{-1}Y_{i+1}Y_{i}^{-1} \\
   = & q^{-1}(Y_{i - 1}Y_{i}^{-1}Y_{j-1}Y_{j}^{-1} + q^{1/2}Y_{i+1}Y_{i}^{-1}Y_{j-1}Y_{j}^{-1}
 + q^{1/2}Y_{i - 1}Y_{i}^{-1}Y_{j+1}Y_{j}^{-1} + qY_{i+1}Y_{i}^{-1}Y_{j+1}Y_{j}^{-1}) \\
   = & q^{-1}X_{2(n - i)}X_{2(n - j)} \\
   = & q^{(-1)^{j - i + 1}}X_{2(n - i)}X_{2(n - j)}
 \end{array}
 \]

 If i is even,j is odd,and $i+1 < j$;

 \[
 \begin{array}{ll}
    & X_{2(n - j)}X_{2(n - i)} \\
   = & (Y_{j - 1} + q^{1/2}Y_{j+1})Y_{j}^{-1}(Y_{i} + q^{1/2}Y_{i + 1})Y_{i}^{-1} \\
   = & Y_{j-1}Y_{j}^{-1}Y_{i - 1}Y_{i}^{-1} + q^{1/2}Y_{j-1}Y_{j}^{-1}Y_{i+1}Y_{i}^{-1}
 + q^{1/2}Y_{j+1}Y_{j}^{-1}Y_{i - 1}Y_{i}^{-1} + qY_{j+1}Y_{j}^{-1}Y_{i+1}Y_{i}^{-1} \\
   = & q(Y_{i - 1}Y_{i}^{-1}Y_{j-1}Y_{j}^{-1} + q^{1/2}Y_{i+1}Y_{i}^{-1}Y_{j-1}Y_{j}^{-1}
 + q^{1/2}Y_{i - 1}Y_{i}^{-1}Y_{j+1}Y_{j}^{-1} + qY_{i+1}Y_{i}^{-1}Y_{j+1}Y_{j}^{-1}) \\
   = & qX_{2(n - i)}X_{2(n - j)} \\
   = & q^{(-1)^{j - i + 1}}X_{2(n - i)}X_{2(n - j)}
 \end{array}
 \]

 So we have $X_{2(n - i)}X_{2(n - j)} = q^{j - i}X_{2(n - j)}X_{2(n - i)}$.

Thus,we get the conclusion.
\end{proof}

\subsection{Example}
The quiver $\Gamma_{2}$ are defined by
\[
\begin{matrix} 1& \leftarrow & 2 &\leftarrow & 3
\end{matrix}.
\]

where
$3$ is the frozen point.

$$\widetilde{B}_{2}=\left(
                      \begin{array}{cc}
                        0 & -1 \\
                        1 & 0 \\
                        0 & 1 \\
                      \end{array}
                    \right)
, \qquad
\Lambda_{2}=\left(
              \begin{array}{ccc}
                0 & -1 & 0 \\
                1 & 0 & 0 \\
                0 & 0 & 0 \\
              \end{array}
            \right)
$$

Since
$$\widetilde{B}_{2}^{T}\Lambda_{2} = \left(
                                      \begin{array}{ccc}
                                        0 & 1 & 0 \\
                                        -1 & 0 & 1 \\
                                      \end{array}
                                    \right)
              \left(
              \begin{array}{ccc}
                0 & -1 & 0 \\
                1 & 0 & 0 \\
                0 & 0 & 0 \\
              \end{array}
            \right)
            =\left(
               \begin{array}{ccc}
                 1 & 0 & 0 \\
                 0 & 1 & 0 \\
               \end{array}
             \right),
$$
  that is, $(\Lambda_{2},\widetilde{B}_{2})$ is compatible.

Let $\A_{n} = \A(\Gamma_{2})$  be the quantum cluster algebra associated with a  pair$(\Lambda_{2},\widetilde{B}_{2})$. By the definition of the quantum cluster algebra,we have $X_{0} = (q^{-1/2}Y_{1} + Y_{3})Y_{2}^{-1}$, \ $X_{2} = Y_{1}^{-1} + q^{1/2}Y_{2}Y_{1}^{-1}$, \ $X_{4} = Y_{1}$, where $(Y_{1}, Y_{2}, Y_{3})$ is the initial quantum cluster.
\begin{proposition}\label{equithm}
 $\A_2$ as a  $\Z[q^{\pm1/2}]$ algebra can be generated by $X_{0}, X_{2}, X_{4}.$
\end{proposition}

\begin{proof}
The detailed expressions of quantum cluster variables are shown as follows:
$Y_{1}, ~Y_{2},~ Y_{3},$~
$Y_{1}^{'}=Y_{1}^{-1} + q^{-1/2}Y_{1}^{-1}Y_{2} = X_{2},~$
$Y_{2}^{'}=q^{1/2}Y_{3}Y_{1}^{-1}Y_{2}^{-1} + Y_{2}^{-1} + Y_{1}^{-1}Y_{3}$,~
$ Y_{1}^{''}=Y_{3}Y_{2}^{-1} + q^{-1/2}Y_{1}Y_{2}^{-1} = X_{0}$.

Because the relation $Y_{1}Y_{2}=q^{-1}Y_{2}Y_{1}$, we have:
$$Y_{1}Y_{1}^{'} = 1+q^{-1/2}Y_{2},$$
$$Y_{1}^{''}Y_{1}^{'} = 1+q^{-1/2}Y_{2}^{'},$$
$$X_{0}Y_{2} = q^{-1/2}Y_{1} + Y_{3}.$$
That is:
$$Y_{2}=q^{1/2}(Y_{1}Y_{1}^{'}-1) = q^{1/2}(X_{4}X_{2} - 1),$$
$$Y_{2}^{'} = q^{1/2}(Y_{1}^{''}Y_{1}^{'} - 1) = q^{1/2}(X_{0}X_{2} - 1),$$
$$Y_{3} = X_{0}Y_{2} - q^{-1/2}X_{4}.$$
So we get the conclusion.

\end{proof}

\begin{theorem}\emph{}\label{thm}
The map
\[
\begin{array}{ccc}
  X_{4} \mapsto [W_{1,4}], & X_{0} \mapsto [W_{1,0}], & X_{2} \mapsto [W_{1,2}]
\end{array}
\]
extends to a ring isomorphism $\iota_{2}$ from the quantum cluster algebra $\A_2$
to the deformed Grothendieck ring $Rep_{t,2}$ of $\CC_2$.
\end{theorem}
\begin{proof}
Since

\[
\begin{array}{ll}
   & Y_{1}^{'}Y_{1}^{''} \\
  = & Y_{1}^{-1}Y_{3}Y_{2}^{-1} + q^{-1/2}Y_{1}^{-1}Y_{2}Y_{3}Y_{2}^{-1}
+ q^{-1}Y_{1}^{-1}Y_{2}Y_{1}Y_{2}^{-1} + q^{-1/2}Y_{2}^{-1}\\
  = & Y_{1}^{-1}Y_{2}^{-1}Y_{3} + q^{-1/2}Y_{1}^{-1}Y_{3} + q^{-1/2}Y_{2}^{-1} +  1
\end{array}
\]
and

\[
\begin{array}{ll}
   & Y_{1}^{''}Y_{1}^{'} \\
  = & Y_{3}Y_{2}^{1}Y_{1}^{-1} + q^{-1/2}Y_{1}Y_{2}^{-1}Y_{1}^{-1}
+ q^{-1/2}Y_{3}Y_{2}^{-1}Y_{1}^{-1}Y_{2} + q^{-1}Y_{1}Y_{2}^{-1}Y_{1}^{-1}Y_{2} \\
  = &  qY_{1}^{-1}Y_{2}^{-1}Y_{3} + q^{1/2}Y_{1}^{-1}Y_{3} + q^{1/2}Y_{2}^{-1} +  1 ,
\end{array}
\]
then we have
$$Y_{1}^{'}Y_{1}^{''} = q^{-1}Y_{1}^{''}Y_{1}^{'} + (1 - q^{-1}).$$

Since
$$Y_{1}^{'}Y_{1} = 1 + q_{1/2}Y_{1}^{-1}Y_{2}Y_{1} = 1 + q^{1/2}Y_{2}$$
and
$$Y_{1}Y_{1}^{'} = 1 + q^{-1/2}Y_{1},$$
then we have
$$Y_{1}Y_{1}^{'} = q^{-1}Y_{1}^{'}Y_{1} + (1 - q^{-1}).$$

Since
$$Y_{1}Y_{1}^{''} = Y_{1}Y_{3}Y_{2}^{-1} + q^{-1/2}Y_{1}^{2}Y_{2}^{-1} = Y_{1}Y_{2}^{-1}Y_{3} + q^{-1/2}Y_{1}^{2}Y_{2}^{-1}$$
and
$$Y_{1}^{''}Y_{1} = Y_{3}Y_{2}^{-1}Y_{1} + q^{-1/2}Y_{1}Y_{2}^{-1}Y_{2} = q^{-1}Y_{1}Y_{2}^{-1}Y_{3} + q^{-3/2}Y_{1}^{2}Y_{2}^{-1},$$
then we have
$$Y_{1}Y_{1}^{''} = qY_{1}^{''}Y_{1}.$$
so we get the conclusion.
\end{proof}

\end{document}